\documentclass[12pt]{amsart}
\usepackage{amssymb}
\usepackage{amsxtra}
\usepackage{graphics}
\usepackage{diagrams}
\usepackage{latexsym}
\usepackage{enumerate}
\usepackage{xcolor}
\usepackage{amsmath}
\usepackage{amssymb,amsthm,amsfonts}
\usepackage{mathtools}
\usepackage{amscd}
\usepackage[arrow, matrix, curve]{xy}
\usepackage[mathscr]{euscript}
\usepackage{syntonly}

\title[ ]{Characterization of Beauville's algebraic numbers via Hodge theory}

\author[Muxi Li]{Muxi Li}

\author[Mao Sheng]{Mao Sheng}

\email{limuxi@ustc.edu.cn}

\email{msheng@ustc.edu.cn}

\address{School of Mathematical Sciences, University of Science and Technology of China, Hefei, 230026, China}

\begin{document}
\theoremstyle{plain}
\newtheorem{thm}{Theorem}[section]
\newtheorem{theorem}[thm]{Theorem}
\newtheorem*{theorem*}{Theorem}
\newtheorem*{theoremA*}{Theorem A}
\newtheorem*{theoremB*}{Theorem B}
\newtheorem*{theoremC*}{Theorem C}
\newtheorem*{definition*}{Definition}
\newtheorem{lemma}[thm]{Lemma}
\newtheorem{sublemma}[thm]{Sublemma}
\newtheorem{corollary}[thm]{Corollary}
\newtheorem*{corollary*}{Corollary}
\newtheorem{prop}[thm]{Proposition}
\newtheorem{addendum}[thm]{Addendum}
\newtheorem{variant}[thm]{Variant}
\theoremstyle{definition}
\newtheorem{lemma-definition}[thm]{Lemma-Definition}
\newtheorem{proposition-definition}[thm]{Proposition-Definition}
\newtheorem{theorem-definition}[thm]{Theorem-Definition}
\newtheorem{construction}[thm]{Construction}
\newtheorem{notations}[thm]{Notations}
\newtheorem{question}[thm]{Question}
\newtheorem{problem}[thm]{Problem}
\newtheorem*{problem*}{Problem}
\newtheorem{remark}[thm]{Remark}
\newtheorem*{remark*}{Remark}
\newtheorem{remarks}[thm]{Remarks}
\newtheorem{definition}[thm]{Definition}
\newtheorem{claim}[thm]{Claim}
\newtheorem{assumption}[thm]{Assumption}
\newtheorem{assumptions}[thm]{Assumptions}
\newtheorem{properties}[thm]{Properties}
\newtheorem{example}[thm]{Example}
\newtheorem{conjecture}[thm]{Conjecture}
\numberwithin{equation}{thm}

\newcommand{\sA}{{\mathcal A}}
\newcommand{\sB}{{\mathcal B}}
\newcommand{\sC}{{\mathcal C}}
\newcommand{\sD}{{\mathcal D}}
\newcommand{\sE}{{\mathcal E}}
\newcommand{\sF}{{\mathcal F}}
\newcommand{\sG}{{\mathcal G}}
\newcommand{\sH}{{\mathcal H}}
\newcommand{\sI}{{\mathcal I}}
\newcommand{\sJ}{{\mathcal J}}
\newcommand{\sK}{{\mathcal K}}
\newcommand{\sL}{{\mathcal L}}
\newcommand{\sM}{{\mathcal M}}
\newcommand{\sN}{{\mathcal N}}
\newcommand{\sO}{{\mathcal O}}
\newcommand{\sP}{{\mathcal P}}
\newcommand{\sQ}{{\mathcal Q}}
\newcommand{\sR}{{\mathcal R}}
\newcommand{\sS}{{\mathcal S}}
\newcommand{\sT}{{\mathcal T}}
\newcommand{\sU}{{\mathcal U}}
\newcommand{\sV}{{\mathcal V}}
\newcommand{\sW}{{\mathcal W}}
\newcommand{\sX}{{\mathcal X}}
\newcommand{\sY}{{\mathcal Y}}
\newcommand{\sZ}{{\mathcal Z}}
\newcommand{\A}{{\mathbb A}}
\newcommand{\B}{{\mathbb B}}
\newcommand{\C}{{\mathbb C}}
\newcommand{\D}{{\mathbb D}}
\newcommand{\E}{{\mathbb E}}
\newcommand{\F}{{\mathbb F}}
\newcommand{\G}{{\mathbb G}}
\newcommand{\HH}{{\mathbb H}}
\newcommand{\I}{{\mathbb I}}
\newcommand{\J}{{\mathbb J}}
\renewcommand{\L}{{\mathbb L}}
\newcommand{\M}{{\mathbb M}}
\newcommand{\N}{{\mathbb N}}
\renewcommand{\P}{{\mathbb P}}
\newcommand{\Q}{{\mathbb Q}}
\newcommand{\R}{{\mathbb R}}
\newcommand{\SSS}{{\mathbb S}}
\newcommand{\T}{{\mathbb T}}
\newcommand{\U}{{\mathbb U}}
\newcommand{\V}{{\mathbb V}}
\newcommand{\W}{{\mathbb W}}
\newcommand{\X}{{\mathbb X}}
\newcommand{\Y}{{\mathbb Y}}
\newcommand{\Z}{{\mathbb Z}}
\newcommand{\id}{{\rm id}}
\newcommand{\rank}{{\rm rank}}
\newcommand{\END}{{\mathbb E}{\rm nd}}
\newcommand{\End}{{\rm End}}
\newcommand{\Hom}{{\rm Hom}}
\newcommand{\Hg}{{\rm Hg}}
\newcommand{\tr}{{\rm tr}}
\newcommand{\Sl}{{\rm Sl}}
\newcommand{\Gl}{{\rm Gl}}
\newcommand{\Cor}{{\rm Cor}}
\newcommand{\Aut}{\mathrm{Aut}}
\newcommand{\Sym}{\mathrm{Sym}}
\newcommand{\ModuliCY}{\mathfrak{M}_{CY}}
\newcommand{\HyperCY}{\mathfrak{H}_{CY}}
\newcommand{\ModuliAR}{\mathfrak{M}_{AR}}
\newcommand{\Modulione}{\mathfrak{M}_{1,n+3}}
\newcommand{\Modulin}{\mathfrak{M}_{n,n+3}}
\newcommand{\Gal}{\mathrm{Gal}}
\newcommand{\Spec}{\mathrm{Spec}}
\newcommand{\res}{\mathrm{res}}
\newcommand{\coker}{\mathrm{coker}}
\newcommand{\Jac}{\mathrm{Jac}}
\newcommand{\HIG}{\mathrm{HIG}}
\newcommand{\MIC}{\mathrm{MIC}}
\newcommand{\SL}{\mathrm{SL}}
\thanks{The work is supported by National Natural Science Foundation of China (Grant No. 11622109, No. 11721101), Chinese Universities Scientific Fund (CUSF) and Anhui Initiative in Quantum Information Technologies (AHY150200)}
\maketitle
\begin{abstract} 
We provide a Hodge theoretical characterization of the set of algebraic numbers which arises from the complete list, due to A. Beauville \cite{Be}, of semistable families of elliptic curves over $\P^1$ with four singular fibers. Our technical innovation is the analysis of the periodicity of the uniformizing Higgs bundle attached to $\P^1$ minus four points over the field of complex numbers.  
\end{abstract}
\section{Introduction}
In the beautiful work \cite{Be}, Beauville gives a complete list of semistable families of elliptic curves over $\P^1$ with four singular fibers. Based on his classification, it is easy to obtain the complete list of complex numbers $\lambda$ such that there is a semistable family of elliptic curves over $\P^1$ with four singular fibers along $\{0,1,\lambda,\infty\}$, that is the following list of algebraic numbers:
\begin{align}\label{beauville's list}
\lambda\in &\{-1,\ 2,\ 1/2,\ -8, \ &&9, \ -1/8,\ 9/8, \ &&1/9, \ 8/9,\\
&(1-\sqrt{-3})/2, &&(1+\sqrt{-3})/2,&& \nonumber \\
&(-123-55\sqrt{5})/2, &&(125+55\sqrt{5})/2, &&(-123+55\sqrt{5})/2, \nonumber \\ 
&(125-55\sqrt{5})/2, &&(25-11\sqrt{5})/50, &&(25+11\sqrt{5})/50.\} \nonumber
\end{align}
We call an algebraic number in the above list a \emph{Beauville's algebraic number}. These numbers have a clear geometric meaning. Namely, these are all possible values such that $\P^1-\{0,1,
\lambda,\infty\}$ is a \emph{modular curve}. Consider an arbitrary $\lambda\in \C$ different from $\{0,1\}$, the fundamental group of $\P^1-\{0,1,\lambda,\infty\}$ is freely generated by three loops. As well-known, the uniformization theorem of Riemann surfaces gives rise to the uniformizing representation which does depend on $\lambda$
$$
\rho_{\lambda}: \pi_1(\P^1-\{0,1,\lambda,\infty\})\to \SL_2(\R). 
$$
When $\lambda$ is a Beauville's algebraic number, $\rho_{\lambda}$ admits a $\Z$-lattice structure. It is well-known that $\rho_{\lambda}$ underlies a weight one polarized $\R$-VHS. Therefore, Beauville's work \cite{Be} amounts to the classification of $\lambda$s such that $\rho_{\lambda}$ underlies a weight one polarized $\Z$-VHS. We may do this a bit better: 
\begin{theorem}\label{main result}
Let $\lambda\neq 0,1$ be a complex number. Then it is a Beauville's algebraic number if and only if the associated uniformizing representation $\rho_{\lambda}$ satisfies the following properties:
\begin{itemize}
	\item [(i)] $\rho_{\lambda}$ factors through $\SL_2(\sO_F)\subset \SL_2(\R)$ for some totally real subfield $F\subset \R$;
	\item [(ii)] $\textrm{Res}_{F|\Q}\rho_{\lambda}$ is a variation of Hodge structure;
	\item [(iii)] There exists one point $x\in \P^1-\{0,1,\lambda,\infty\}$ such that the multiplication by any element of $\sO_F$ is of Hodge type $(0,0)$. 
\end{itemize}
\end{theorem}  
We also intend to give another characterization of Beauville's algebraic numbers. Recall that in the Hitchin-Simpson's approach to the uniformization theory of Riemann surfaces, one studies the so-called uniformizing Higgs bundle $(E_{\lambda},\theta_{\lambda})$ which is constructed as follows:
$$
E_{\lambda}=\sO_{\P^1}(1)\oplus \sO_{\P^1}(-1),\quad \theta_{\lambda}=\theta_{\lambda}^{1}\oplus \theta_{\lambda}^{0},
$$
where $\theta^{0}=0$ and $\theta^{1}: \sO_{\P^1}(1)\to \sO_{\P^1}(-1)\otimes \Omega_{\P^1}(0+1+\lambda+\infty)$ is an isomorphism. Under the correspondence established in \cite{Si}, the complex local system $\rho_{\lambda}\otimes \C$ and the logarithmic Higgs bundle $(E_{\lambda},\theta_{\lambda})$ correspond to each other. 
\begin{remark}
Using the theory of tame harmonic bundles \cite{Si}, one should be able to provide an obvious analogue of Lemma 2.11 \cite{Si92} in the quasi-projective curve case. So one may argue that $\rho_{\lambda}$ corresponds to $(E_{\lambda},\theta_{\lambda})$ as \emph{real} local system. 
\end{remark} 
Based on the work of \cite{LSZ}, we connect the $\Z$-lattice structure of $\rho_{\lambda}$ with the \emph{periodicity} of $(E_{\lambda},\theta_{\lambda})$. More precisely, we have the following
\begin{prop}[Proposition \ref{beauville's list implies periodicity}]
For any Beauville's algebraic number $\lambda$, $(E_{\lambda},\theta_{\lambda})$ is periodic. 
\end{prop}
For the notion of a \emph{periodic Higgs bundle} over $\C$, we refer our reader to Definition \ref{periodic definition}. Furthermore, we show the following statement. 
\begin{theorem}[Proposition \ref{one periodicity}, Proposition \ref{transcendental number}]
If $(E_{\lambda},\theta_{\lambda})$ is periodic, then it must be one-periodic and $\lambda$ must be algebraic. 
\end{theorem} 
We may call an algebraic number $\lambda$ \emph{periodic} if $(E_{\lambda},\theta_{\lambda})$ is periodic. The set of periodic algebraic numbers contains the Beauville's list \ref{beauville's list}. We make the following
\begin{conjecture}
The Beauville's algebraic numbers are \emph{all} periodic algebraic numbers. 
\end{conjecture}
A brute-force calculation on the periodicity reduces the conjecture to a concrete arithmetic question (see Question \ref{arithmetic question}). However, solving it is beyond our ability. So we would like to leave it for inspired readers. 

The paper is structured as follows: In Section 2, we introduce the notion of a periodic Higgs bundle over the field of complex numbers. We prove that a logarithmic Higgs subbundle of degree zero in the associated logarithmic Kodaira-Spencer system to a semistable family is periodic. In Section 3 we give a detailed analysis of the periodicity of the uniformizing Higgs bundle in both positive characteristic and zero characteristic. In Section 4, we construct a semistable family of abelian varieties with real multiplication under the Hodge theoretical properties listed in Theorem \ref{main result}. This is a straightforward step. In Section 5, we prove our main result Theorem \ref{main result}. 

{\bf Acknowledgement.} We would like to thank heartily Professor Kang Zuo for valuable discussions in the earlier formulation of Theorem \ref{main result}. The problem considered in this note stems from the joint work of the second named author with him \cite{LSYZ} (see particularly Proposition 3.7 loc. cit.). We would like also to mention his recent work \cite{Zuo} which contains a beautiful conjecture on the periodicity of \emph{another} rank two Higgs bundle over $\P^1$ removing four points. Also, in the forthcoming work \cite{KYZ}, Krishnamoorthy, Yang and Zuo shall provide a characterization of modular curves via $p$-adic periodic Higgs-de Rham flow.  

\section{Logarithmic Kodaira-Spencer systems and periodicity}
Let $C$ be a smooth projective curve over $\C$ and $D\subset C$ be a reduced effective divisor. Let $f: X\to C$ be a semistable family which is smooth over $U=C-D$. Set $B=f^{-1}D$. We attach to $f$ the \emph{logarithmic Kodaira-Spencer system} of degree $n$ for some $0\leq n\leq 2d$ (where $d$ is the relative dimension of $f$) as follows: 
$$
(E=\bigoplus_{i+j=n}R^if_{*}\Omega^j_{X/C}(\log B/D), \theta=\bigoplus_{i+j=n}\theta^{i,j}),
$$
where 
$$
\theta^{i,j}: R^if_{*}\Omega^j_{X/C}(\log B/D)\to R^{i+1}f_{*}\Omega^{j-1}_{X/C}(\log B/D)\otimes \Omega_{C}(D)
$$
is the logarithmic Kodaira-Spencer morphism. Set $E^{i,j}=R^if_{*}\Omega^{j}_{X/C}(\log B/D)$. The pair $(E,\theta)$ provides a basic example of graded logarithmic Higgs bundle over $C$. These logarithmic Kodaira-Spencer systems are special kinds of graded logarithmic Higgs bundles, as the following result shows:
\begin{prop}\label{polystable}
Let $(E,\theta)$ be a logarithmic Kodaira-Spencer system over $C$ as above. Then it is polystable of degree zero.
\end{prop}
\begin{proof}
This is well-known from the nonabelian Hodge theory over a quasi-projective smooth curve \cite{Si}. Let $f^0$ be the smooth part of $f$. Then the Hodge metric associated to the weight $n$ VHS attached to $f^0$ is \emph{tame harmonic}. By Landman's theorem, the local monodromies around $D$ are unipotent. It follows that the filtration structures are absent in the Simpson's correspondence (Main Theorem page 755 \cite{Si}).  So the statement follows.
\end{proof}
The result above has many deep implications in the geometry of fibrations. But by the transcendental nature of the method of establishing the Simpson's correspondence (which refers to the Main Theorem \cite{Si}), the lattice structure underlying the $\Z$-VHS associated to a family seems undetectable in the associated logarithmic Kodaira-Spencer system. In this note, we are trying to argue that some dynamical property associated to various mod $p$ reductions of the logarithmic Kodaira-Spencer system does bear information from the lattice structure. However, our finding into this issue is far from being definite. 

Now let $k$ be a perfect field of positive characteristic $p$. Let $X$ be a smooth variety over $k$ and $D\subset X$ a simple normal crossing divisor. Set $X_{\log}=(X,D)$ to be the logarithmic variety over $k$ (regarded as a log scheme with trivial log structure) whose log structure determined by the divisor $D$. Using the logarithmic generalization of the inverse Cartier transform of Ogus-Vologodsky due to D. Schepler (see also Appendix \cite{LSZ0}), one may generalize Definition 1.1 \cite{LSZ} in the straightforward manner:
\begin{definition} 
Let $X,D$ be as above. A \emph{periodic Higgs-de Rham flow} over $X_{\log}$ of period $f\in \N_{>0}$ is a diagram as follows:

$$
\xymatrix{
	& C^{-1}(E_0,\theta_0)\ar[dr]^{Gr_{Fil_0}} && C^{-1}(E_{f-1},\theta_{f-1})\ar[dr]^{Gr_{Fil_{f-1}}} \\
	(E_0,\theta_0) \ar[ur]^{C^{-1} } & & \cdots \ar[ur]^{C^{-1}}&& (E_f,\theta_f),\ar@/^2pc/[llll]^{\stackrel{\phi}{\cong} } }
$$

where

\begin{itemize}
	
	\item the initial term $(E_0,\theta_0)$ is a nilpotent logarithmic graded Higgs bundle with pole along $D$ and of level $\leq p-1$,
	
	\item for each $i\geq 0$ $Fil_i$ is a Hodge filtration on the logarithmic flat bundle $C^{-1}(E_i,\theta_i)$ making it into a de Rham bundle of level $\leq p-1$,
	
	\item $(E_i,\theta_i),i\geq 1$ is the associated logarithmic graded Higgs bundle to the logarithmic de Rham bundle $(C^{-1}(E_{i-1},\theta_{i-1}),Fil_{i-1})$,
	
	\item and $\phi$ is an isomorphism of logarithmic graded Higgs bundles.
	
\end{itemize}

\end{definition}
A logarithmic graded Higgs bundle over $X$ with pole along $D$ is said to be \emph{periodic} if it initializes a periodic Higgs-de Rham flow (of certain period $f$) over $X_{\log}$. We turn our attention back to the situation over $\C$. 
\begin{definition}
Let $X$ be a smooth variety over $\C$ and $D\subset X$ a SNCD. Set $X_{\log}=(X,D)$. Let $(E,\theta)$ be a graded Higgs bundle over $X_{\log}$ (namely a graded logarithmic Higgs bundle over $X$ with pole along $D$). A \emph{spread} of $(X,D,E,\theta)$ is a four-tuple $(\sX,\sD, \sE,\Theta)$ where $\sX$ is a scheme of finite type over $S=\Spec(A)$ with $A\subset \C$ a finitely generated $\Z$-subalgebra, $\sD$ is an $S$-relative divisor in $\sX$ and 
$$
\Theta: \sE\to \sE\otimes \Omega_{\sX/S}(\log \sD)
$$
an $S$-relative logarithmic Higgs bundle over $\sX$ such that there is an isomorphism of tuples over $\C$:
$$
(\sX,\sD,\sE,\Theta)\otimes _{A}\C\cong (X,D,E,\theta).
$$ 
\end{definition}
We give two lemmas which will be used below.
\begin{lemma}\label{black hole}
Let $k$ be an algebraically closed field of characteristic $p$. Let $C$ be a smooth projective curve over $k$ and $D\subset C$ a reduced effective divisor. If the initial term of a periodic Higgs-de Rham flow over $C_{\log}$ is stable, then any intermediate Higgs term of the periodic flow is also stable.
\end{lemma}
\begin{proof}
For a given semistable logarithmic Higgs bundle over $C_{\log}$, its Jordan-H\"{o}lder filtration may not be unique, but different Jordan-H\"{o}lder filtrations have the same set of stable factors. Let $\{E_0,\cdots,E_{f-1}\}$ be the intermediate Higgs terms of an $f$-periodic Higgs-de Rham flow with $E_f\cong E_0$. Set $n_i$ to be the number of stable factors of a Jordan-H\"{o}lder filtration on $E_i$. We claim that $n_i\leq n_{i+1}$.  Indeed, we are going to show that the flow operator $Gr_{Fil}\circ C^{-1}$ maps a Jordan-H\"{o}lder filtration on $E_i$ to a filtration of semistable subbundles on $E_{i+1}$ whose length equals $n_i$, from which the claim follows. By Proposition 6.3 \cite{LSZ}, each $E_i$ is of degree zero. So each member in a Jordan-H\"{o}lder filtration of $E_i$ is of degree zero too. Because the operator $C^{-1}$ is exact and multiplies the degree by $p$, one obtains consequently a filtration on $C^{-1}(E_i)$ of flat subbundles of degree zero. As $Gr_{Fil}$ preserves the degree, using the induced filtrations and take the associated gradings, one obtains a filtration of the same length by Higgs subsheaves of degree zero on $Gr_{Fil}\circ C^{-1}(E_i)=E_{i+1}$. It is actually a filtration of Higgs \emph{subbundles}. This is because the saturation of each Higgs subsheave is invariant under the Higgs field and hence of degree $\leq 0$ by semistability. The claim is proved. It follows that
$$
	n_0\leq n_1\leq\cdots\leq n_f=n_0,
$$
hence all the $n_i$s are equal and the flow operator maps a Jordan-H\"{o}lder filtration of $E_i$ to a Jordan-H\"{o}lder filtration of $E_{i+1}$. It follows immediately that, if $E_0$ is stable, then each intermediate Higgs term is also stable. This concludes the proof.
\end{proof}

\begin{lemma}\label{uniqueness on filtration}
	Let $C,D$ be as Lemma \ref{black hole}. Let $(V,\nabla)$ be a flat bundle over $C_{\log}$. Then up to a shift of indices, there is at most one Hodge filtration $Fil$ on $(V,\nabla)$ such that the graded Higgs bundle $Gr_{Fil}(V,\nabla)$ is stable.  
\end{lemma}
\begin{proof}
The proof does not make difference for $D$ empty or not. For $D$ empty, see Lemma 4.1 \cite{LSZ}. 
\end{proof}
We introduce the following definition.
\begin{definition}\label{periodic definition}
	Let $(X,D)$ be a log pair over $\C$. A graded Higgs bundle $(E,\theta)$ over $X_{\log}$ is called \emph{periodic} if there exists a spread $$(\sX,\sD, \sE,\Theta)\rightarrow S,$$ of $(X,D,E,\theta)$, a positive integer $f$ and a closed subset $Z\subset S$ which has finite image in $\Spec (\Z)$ such that for all geometric points $s\in S-Z$, the reduction $(\sE_s,\Theta_s)$ at $s$ is periodic of period $\leq f$ for \emph{all} $W_2(k(s))$-lifting $\tilde s\to S-Z$. The \emph{period} of a periodic Higgs bundle is defined to be the smallest $f\geq 1$ among all possible spreads.  
\end{definition}
The following result is immediate by the theory of spread.
\begin{prop}
Let $X,D$ and $(E,\theta)$ be as in Definition \ref{periodic definition}. Then if $(E,\theta)$ is periodic with respect to one spread, then it is periodic for any spread. Hence, periodicity of $(E,\theta)$ is an intrinsic property over $\C$. 
\end{prop}
We list several simple properties of periodic Higgs bundles.
\begin{lemma}\label{basic properties of periodic Higgs bundles}
The following statements about periodic Higgs bundles hold:
\begin{itemize}
	\item [(i)] A periodic Higgs bundle is semistable of degree zero.
	\item [(ii)] A direct sum of periodic Higgs bundles is again periodic.
	\item [(iii)] A graded Higgs bundle obtained by renumbering the graded structure of a periodic Higgs bundle is again periodic.
\end{itemize}
\end{lemma}
\begin{proof}
(i) follows from Proposition 6.3 \cite{LSZ}. (ii) is obvious. As for (iii), one notices that the inverse Cartier transform ignores the grading structure. So in char $p$, if it is periodic for one grading structure, then for another grading structure (as long as its the largest grading is $\leq p-1$) one may simply adjust the indices of the last Hodge filtration to make it periodic as well.
\end{proof}
The main result of this section is the following periodicity result.    
\begin{theorem}\label{periodicity theorem}
Let $C$ be a smooth projective curve and $D\subset C$ a reduced effective divisor. Let $f: X\to C$ be a semistable family and $(E,\theta)$ be a logarithmic Kodaira-Spencer system associated to $f$. Then any graded Higgs subbundle in $(E,\theta)$ of degree zero is periodic in the sense of Definition \ref{periodic definition}.
\end{theorem}
\begin{proof}
We take a spread of the family $f$ as follows. By the standard argument (\cite{EGA IV} 8, 11.2, 17.7), there exists a sub $\Z$-algebra $A\subset \C$ of finite type and a semistable family $\mathfrak{f}: (\mathfrak X, \mathfrak {B})\to (\mathfrak C, \mathfrak D)$ defined over $S=\Spec(A)$ such that $S$ is integral and regular, $\alpha: \mathfrak C\to S$ is smooth and projective, and $f$ is the base change of $\mathfrak f$ via $\Spec(\C)\to S$. Shrinking $S$ if necessary, we may assume that for any closed point $s\in S$,
$$\mathrm{char}(k(s))>N:=\mathrm{rank}(E)+n+d+1.
$$
By Deligne-Illusie \cite{DI}, Illusie \cite{IL90}, for any $i,j$, $$R^{j}\mathfrak{f}_{*}\Omega^{i}_{\mathfrak{X}/\mathfrak{C}}(\log \mathfrak{B}/\mathfrak{D}),~ R^{i+j}\mathfrak{f}_{*}\Omega^{*}_{\mathfrak{X}/\mathfrak{C}}(\log \mathfrak{B}/\mathfrak{D})$$ are locally free of finite type, and the spectral sequence
\begin{eqnarray*}\label{int E1}E^{i,j}_1=R^{j}\mathfrak{f}_{*}\Omega^{i}_{\mathfrak{X}/\mathfrak{C}}(\log \mathfrak{B}/\mathfrak{D})\Rightarrow R^{i+j}\mathfrak{f}_{*}\Omega^{*}_{\mathfrak{X}/\mathfrak{C}}(\log \mathfrak{B}/\mathfrak{D})\end{eqnarray*}
degenerates at $E_1$. It yields the logarithmic Kodaira-Spencer system associated to $\mathfrak f$. 
$$
\Theta^{i,j}: R^{j}\mathfrak{f}_{*}\Omega^{i}_{\mathfrak{X}/\mathfrak{C}}(\log \mathfrak{B}/\mathfrak{D})\to R^{j+1}\mathfrak{f}_{*}\Omega^{i-1}_{\mathfrak{X}/\mathfrak{C}}(\log \mathfrak{B}/\mathfrak{D})\otimes \Omega_{\mathfrak{C}/S}(\log \mathfrak D)\footnote{Another way to obtain $\Theta^{i,j}$ is obtained by taking a suitable edge morphism associated to the higher direct image of the following short exact sequence:
	$$
	0\to \mathfrak{f}^*\Omega_{\mathfrak{C}/S}(\log \mathfrak D)\otimes \Omega^{i-1}_{\mathfrak{X}/\mathfrak{C}}(\log \mathfrak{B}/\mathfrak{D})\to \Omega^{i}_{\mathfrak{X}}(\log \mathfrak{B})\to \Omega^{i}_{\mathfrak{X}/\mathfrak{C}}(\log \mathfrak{B}/\mathfrak{D})\to 0.
	$$}.
$$
Set $\sE^{i,j}=R^{j}\mathfrak{f}_{*}\Omega^{i}_{\mathfrak{X}/\mathfrak{C}}(\log \mathfrak{B}/\mathfrak{D})$. Thus $\sE=\bigoplus_{i+j=n}\sE^{i,j},\Theta=\bigoplus_{i+j=n}\Theta^{i,j}$ is a spread of $(E,\theta)$. By construction, for any geometrically closed point $s\in S$, the base change $(\sE_s,\Theta_s)$ is the logarithmic Kodaira-Spencer system associated to the family $\mathfrak f_s$, the fiber of $\mathfrak f$ at $s$. By Theorem 6.2 \cite{Fa89} and Proposition 4.1 \cite{LSZ0}, $(\sE_s,\Theta_s)$ is one periodic with respect to any $W_2(k(s))$-lifting $\tilde s: \Spec(W_2(k(s)))\to S$.

By Proposition \ref{polystable}, we may write $(E,\theta)$ into direct sum of stable factors:
$$
(E,\theta)=\bigoplus_{i=1}^{l}(E_i,\theta_i)^{\oplus m_i}
$$
with $(E_i,\theta_i)$ stable. Caution: the above equality is meant to be an equality of Higgs bundles, instead of graded Higgs bundles over $C_{\log}$; for different $i$ and $j$, $(E_i,\theta_i)$ are not isomorphic to $(E_j,\theta_j)$ as Higgs bundles.    

It suffices to show that $(E_i,\theta_i)$ is periodic for each $i$. This is because any graded Higgs subbundle of degree zero in $(E,\theta)$ must be a direct sum of stable factors $(E_i,\theta_i)$s, up to renumbering the graded structure. Then one applies Lemma \ref{basic properties of periodic Higgs bundles}. As the geometric stability is an open condition, shrinking $S$ if necessary, we may assume that 
$$
(\sE,\Theta)=\bigoplus_{i=1}^{l}(\sE_i,\Theta_i)^{\oplus m_i}
$$
such that for each $i$, $(\sE_i,\Theta_i)$ is a spread of $(E_i,\theta_i)$ and its base change at $s$ is still stable. So we obtain the following decomposition into stable factors: 
$$
(\sE_s,\Theta_s)=\bigoplus_{i=1}^{l}(\sE_{i,s},\Theta_{i,s})^{\oplus m_i}.
$$
Let $Fil$ denote the Hodge filtration of the family in consideration. Then the one-periodicity of $(E,\theta)$ means an isomorphism:
$$
Gr_{Fil}\circ C^{-1}(\sE_s,\Theta_s)\cong (\sE_s,\Theta_s).
$$
We argue that the operator $Gr_{Fil}\circ C^{-1}$ induces a self-map the set $T=\{1,\cdots,l\}$ which represents the set of non-isomorphic stable factors in $(\sE_s,\Theta_s)$. Pick any $i\in T$. By Lemma \ref{black hole}, $Gr_{Fil}\circ C^{-1}(\sE_{i,s},\Theta_{i,s})$ is stable. Therefore, there is a unique $j(i)$ such that 
$$
Gr_{Fil}\circ C^{-1}(\sE_{i,s},\Theta_{i,s})\cong (\sE_{j(i),s}, \Theta_{j(i),s})
$$ 
as logarithmic Higgs bundles. However, when $m_i\geq 2$, three are more than one factor isomorphic to $(\sE_{i,s},\Theta_{i,s})$. We claim that $j(i)$ does not depend on this ambiguity. Let $(\sE_{i,s},\Theta_{i,s})_1\cong (\sE_{i,s},\Theta_{i,s})_2$ be two stable factors of $(\sE_s,\Theta_s)$ which are isomorphic as logarithmic Higgs bundles. Since $C^{-1}$ is an equivalence of categories,  
$$
C^{-1}(\sE_{i,s},\Theta_{i,s})_1\cong C^{-1}(\sE_{i,s},\Theta_{i,s})_2.
$$  
as logarithmic flat bundles. As $Gr_{Fil}C^{-1}(\sE_{i,s},\Theta_{i,s})_i, i=1,2$ is stable, it follows from Lemma \ref{uniqueness on filtration} that there is an isomorphism of logarithmic Higgs bundles:
$$
Gr_{Fil}\circ C^{-1}(\sE_{i,s},\Theta_{i,s})_1\cong Gr_{Fil}\circ C^{-1}(\sE_{i,s},\Theta_{i,s})_2.
$$
Hence the claimed independence holds. So we get a well-defined map
$$
Gr_{Fil}\circ C^{-1}: T\to T.
$$
This map has to be surjective, because the rank is preserved under the flow operator. As $T$ being a finite set, it is bijective and therefore decomposes into a product of cyclic permutations. Thus, for each $i$, the Hodge filtration $Fil$ induces an $f$-periodic flow with initial term $(\sE_{i,s},\Theta_{i,s})^{\oplus m_i}$ for some $f\leq l!$. It induces in turn an $f$-periodic flow with initial term for any factor 
$$
(\sE_{i,s},\Theta_{i,s})\subset (\sE_{i,s},\Theta_{i,s})^{\oplus m_i}\subset (\sE_{s},\Theta_{s}).
$$
This completes the whole proof. 
\end{proof}
      
\section{Periodicity of the uniformizing Higgs bundle}
In this section, we shall investigate into the periodicity of the uniformizing Higgs bundle over $\P^1$ with four simple poles, both in characteristic $p>0$ and in characteristic zero. As a matter of convention, we shall use the notation $(E_{unif},\theta_{unif})$ for the uniformizing Higgs bundle over $\P^1$ with four simple poles over an \emph{arbitrary} field.

Let us start with the periodicity over $\C$.   
\begin{prop}\label{one periodicity}
Let $(E_{unif},\theta_{unif})$ be the uniformizing Higgs bundle over $(\P^1,D)$ over $\C$, where $D$ consists of four distinct points. If $(E_{unif},\theta_{unif})$ is periodic, its period is equal to one. 
\end{prop}
\begin{proof}
Let $(\P^1,\sD,\sE_{unif},\Theta_{unif})$, defined over $S$, be a spread of $(\P^1,D, E_{unif},\theta_{unif})$. Let $s\in S$ be a geometrically closed point. Shrinking $S$ if necessarily, we may assume $(\sE_{unif,s},\Theta_{unif,s})$ is isomorphic to $(\sO(1)\oplus \sO(-1),id)$ and hence stable of trivial determinant. Let $(L\oplus L^{-1}, \theta)$ be a periodic Higgs bundle over $(\P_s^1,\sD_s)$ with trivial determinant. We claim that if it is stable, then it must be isomorphic to $(\sE_{unif,s},\Theta_{unif,s})$. Indeed, because the Higgs field induces a nonzero morphism $L^{\otimes 2}\to \Omega_{\P_s^1}(\sD_s)\cong \sO(2)$, it follows that $\deg L\leq 2$. As $\deg L\geq 0$ in any case, one has
$$
\deg L=0, 1.
$$
But if $\deg L=0$, we have the Higgs subbundle $(\sO_{\P^1},0)\subset (L\oplus L^{-1}, \theta)$ which violates the stability. Hence $\deg L=1$ and $\theta$ must be an isomorphism for the degree reason. In other words, $(L\oplus L^{-1}, \theta)\cong (\sE_{unif,s},\Theta_{unif,s})$. By Lemma \ref{black hole}, any intermediate Higgs terms of a periodic flow initializing $(\sE_{unif,s},\Theta_{unif,s})$ is periodic, stable and trivial determinant (which is clear). The proposition follows.  
\end{proof}
The pair $(\P^1,D)$ in Proposition \ref{one periodicity} is isomorphic to $(\P^1, 0+1+\lambda+\infty)$ for some $\lambda\neq 0,1$.  
\begin{prop}\label{beauville's list implies periodicity}
If $\lambda$ belongs to the Beauville's list \ref{beauville's list}, then $(E_{unif},\theta_{unif})$ is periodic. 
\end{prop}
\begin{proof}
Let $\lambda$ be a value in the Beauville's list. Beauville shows that there is a semistable family of elliptic curves over $(\P^1, D=0+1+\lambda+\infty)$. The family is non-isotrivial, hence the associated period map is nonconstant. It follows that the associated logarithmic Kodaira-Spencer system $(E,\theta)$ has nonzero Higgs field (one may also see this by the existence of singular fibers).  In fact, $(E,\theta)$ is isomorphic to $(E_{unif},\theta_{unif})$. By Proposition \ref{polystable}, $(E,\theta)$ takes the form
$$
\theta: E^{1,0}\to E^{0,1}\otimes \Omega_{\P^1}(D)
$$
with $E^{0,1}$ isomorphic to the dual of $E^{1,0}$ and $\deg E^{1,0}>0$. On the other hand, since $\theta$ is nonzero, we have 
$$
\deg (E^{1,0})^{\otimes 2}\leq \deg \Omega_{\P^1}(D)=2.
$$ 
Thus $\deg E^{1,0}=1$. By Theorem \ref{periodicity theorem}, the proposition follows. 
\end{proof}
We conjecture the converse of Proposition \ref{beauville's list implies periodicity}.
\begin{conjecture}\label{conjecture}
Notation as in Proposition \ref{beauville's list implies periodicity}. If $(E_{unif},\theta_{unif})$ is periodic, then $\lambda$ must be in Beauville's list.
\end{conjecture}
We may give a very partial answer of the conjecture.
\begin{prop}\label{transcendental number}
If $(E_{unif},\theta_{unif})$ over $(\P^1,0+1+\lambda+\infty)$ is periodic, then $\lambda$ is algebraic. 
\end{prop}
We shall give a proof of the last statement below. In order to approach the problem in the conjecture, a detailed study of periodicity in positive characteristic is necessary. So we let $k=\bar{\mathbb{F}}_{p}$, and intend to make an explicit study of the periodicity condition for $(E_{unif},\theta_{unif})$ over $k$. 

Let $\lambda=(\lambda_0,\lambda_1)\in W_2(k)$ with $\lambda_0$ distinct from $\{0,1\}$. This $\lambda$ gives rise to an obvious $W_2(k)$-lifting of the pair $(\P^1,0+1+\lambda_0+\infty)$ over $k$. If one fixes a $\lambda_0\in k$, then any $W_2(k)$-lifting of the pair  $(\P^1,0+1+\lambda_0+\infty)$ is isomorphic to the one given by some $\lambda\in W_2(k)$ as above. We define a $p\times (2p+1)$ matrix $T$ as follows:

\[
T_{ij}=
\begin{cases}
\lambda_1 & \text{if $i=j$}\\
(-1)^{i-j+1}\frac{{{p}\choose{p-i+j}}}{p}(1-\lambda_0^{i-j}) & \text{if $i>j$}\\
(-1)^{j-i+1}\frac{{{p}\choose{p-j+i}}}{p}(\lambda_0^{p-j+i}-\lambda_0^p) & \text{if $i<j\leq p$}\\
(-1)^{i+j-p-1}\frac{{{p}\choose{i+j-p-1}}}{p}(1-\lambda_0^{i+j-p-1}) & \text{if $p<j\leq 2p-i$}\\
0 & \text{if $j>2p-i$}
\end{cases}
\]
Take $T_m$ as the $(p-m)\times(p+m)$ submatrix of $T$ containing the first $p-m$ rows and first $p+m$ columns ($0\leq m\leq p-1$). We obtain the following result:
\begin{prop}\label{condition for periodicity in char p}
Let $k=\bar{\F}_p$ and $\lambda_0\in k$ be an element distinct from $\{0,1\}$. Then $(E_{unif},\theta_{unif})$ over $\mathbb{P}^1-\{0,1,\lambda_0,\infty\}$ with respect to the $W_2(k)$-lifting given by  $\lambda:=(\lambda_0,\lambda_1)\in W_2(k)$ is periodic if and only if 
$$
\det(T_0)=0;\quad \rank(T_1)=p-1.
$$
\end{prop}
Let us first proceed to the proof of Proposition \ref{transcendental number}.
\begin{proof}
Assume the contrary. So if $\lambda\in \C$ is transcendental, then a spread of $(\P^1,0+1+\lambda+\infty)$ is defined over $S=\Spec(\Z[\lambda])$ which is isomorphic to the affine line $\A^1_{\Z}$ over $\Z$. By Proposition \ref{condition for periodicity in char p}, for $\lambda_0\in \bar{\F}_p$ which is considered as a geometrically closed point of $S$, there are at most $p$ $\lambda_1$s such that $(E_{unif},\theta_{unif})$ over $(\P^1,0+1+\lambda_0+\infty)$ is periodic with respect to the $W_2$-lifting determined by $\lambda=(\lambda_0,\lambda_1)$. Therefore, at any rate, the uniformizing Higgs bundle over $(\P^1,0+1+\lambda+\infty)$ cannot be periodic. Contradiction.  
\end{proof}
\begin{remark}
One may perhaps further explore into the implication of Proposition \ref{condition for periodicity in char p}. Let $\lambda\in \sO_K[\frac{1}{N}]$ be an algebraic number, where $K\subset \C$ is an algebraic number field and $N$ a natural number. Then for almost all places $\nu\in K$, $K$ is unramified at $\nu$ and therefore $\sO_K/\nu^2\cong W_2(\F_q)\subset W_2(\bar{\F}_p)$. Fix \emph{any} such morphism $r_{\nu}=(r_{\nu,0},r_{\nu,1})$. Conjecture \ref{conjecture} amounts to the truth of the following arithmetic
\begin{question}\label{arithmetic question}
Let $\lambda\in K$ be as above. Assume for almost all places $\nu$ of $K$, the pair $(r_{\nu,0}(\lambda), r_{\nu,1}(\lambda))$ satisfies 
$$
\det T_0(r_{\nu,0}(\lambda), r_{\nu,1}(\lambda))=0,\quad \rank \ T_1(r_{\nu,0}(\lambda), r_{\nu,1}(\lambda))=\textrm{char}(k(\nu))-1.
$$
Is $\lambda$ a Beauville's algebraic number?	
\end{question}  
\end{remark}
Now we turn to the proof of Proposition \ref{condition for periodicity in char p}. It relies on the following analysis.
\begin{lemma}\label{inverse Cartier for periodicity}
	$(E_{unif},\theta_{unif})$ over $\mathbb{P}^1-\{0,1,\lambda_0,\infty\}$ with respect to some $W_2(k)$-lifting determined by $\lambda=(\lambda_0,\lambda_1)$ is one periodic if and only if the bundle part of the inverse Cartier transform of $(E_{unif},\theta_{unif})$ with respect to that lifting is isomorphic to $\mathcal{O}(-1)\oplus \mathcal{O}(1)$.
\end{lemma}
\begin{proof}
Write $(H,\nabla)=C^{-1}(E_{unif},\theta_{unif})$, where $C^{-1}$ refers to the inverse Cartier transform with respect to the $W_2(k)$-lifting determined by $\lambda$. Assume for this $W_2$-lifting, $(E_{unif},\theta_{unif})$ is one-periodic. Then one has an short exact sequence
	$$
	0\to \mathcal{O}(1)\to H\to \mathcal{O}(-1)\to 0.
	$$
	Computing that $\dim \mathrm{Ext}^1(\mathcal{O}(-1),\mathcal{O}(1))=h^1(\mathcal{O}(2))=h^0(\mathcal{O}(-4))=0$, we find that $H$ must be isomorphic to $\mathcal{O}(-1)\oplus \mathcal{O}(1)$. Conversely, let us assume that $H\cong \mathcal{O}(-1)\oplus \mathcal{O}(1)$. Set $\mathcal{O}(1)\cong Fil^1\subset H$. Note that $Fil^1$ cannot be $\nabla$-invariant for the degree reason. Thus, the graded Higgs field $Gr_{Fil}\nabla$ must be nonzero. Again for the degree reason, it must be maximal, that is, one has 
	$$
	Gr_{Fil}\nabla: Fil^1\cong H/Fil\otimes \Omega^1_{\mathbb{P}^1}(\log 0+1+\lambda_0+\infty).
	$$
	This completes the lemma. 
\end{proof}

The next step is then to determine $(H,\nabla)=C^{-1}(E_{unif},\theta_{unif})$, especially the bundle part $H$. Note that there exists a unique natural number $n$ such that $H\cong \sO(n)\oplus \sO(-n)$. Our main goal in the following is to determine the $n$. We do this via the approach of exponential twisting to the inverse Cartier transform (see \cite{LSZ0}, \cite{LSYZ} Appendix). Set $\tilde X=\mathbb{P}^1-\{0,1,\tilde \lambda, \infty\}$ over $W_2$ and $X$ to be its reduction. The curve $X$ has a distinguished open affine covering $\mathcal{U}=\{U_{\alpha},U_{\beta}\}$ with
$$
U_{\alpha}=\mathbb{P}^1-\{0,\infty\};\quad U_{\beta}=\mathbb{P}^1-\{1,\lambda_0\}.
$$
Set $\tilde U_{\alpha}\subset \tilde X$ to be the open affine scheme by restricting $\tilde X$ to $U_{\alpha}$ and similarly define the open subscheme $\tilde U_{\beta}\subset \tilde X$. Let $z$ be the affine coordiante of $\mathbb{P}^1-\{\infty\}$. Then, one may choose the standard log Frobenius lifting determined by
$$
\tilde F_{\alpha}(z)=z^p; \quad \tilde F_{\beta}(w)=w^p,
$$
where $w=\frac{z-\lambda}{z-1}$ ($w$ is a linear transformation of $\mathbb{P}^1_{W_2}$ which maps $\lambda$ to zero and 1 to $\infty$). On the overlap $U_{\alpha\beta}=U_{\alpha}\cap U_{\beta}$, we use the coordinate $z$. Therefore, the second Frobenius lifting $\tilde F_{\beta}$ on $\tilde U_{\alpha}$ 
is written as
$$
z\mapsto \frac{(z-\lambda)^p-F(\lambda)(z-1)^p}{(z-\lambda)^p-(z-1)^p},
$$
where $F(\lambda)=(\lambda_0^p,\lambda_1^p)$. Here is a double check: RHS mod $p$ is nothing but $z^p$. By the definition of a log Frobenius lifting, RHS
can be written as $z^p(1+pa)$ with 
$$
a=\frac{1}{p}\cdot \frac{(z^p-F(\lambda))(z-1)^p-(z-\lambda)^p(z^p-1)}{z^p[(z-\lambda)^p-(z-1)^p]}\in k[z,\frac{1}{z(z-1)(z-\lambda_0)}]=\mathcal{O}_{U_{\alpha\beta}}
$$ 
(Notice that $F(\lambda)=(\lambda_0^p,0)+(0,\lambda_1^p)=\lambda^p+p(\lambda_1,0)$, which means the numerator is divisible by $p$).

Actually, we can write $a$ more precisely into:

$$
a=\frac{\sum_{i=1}^{p-1}(-1)^i\frac{{p \choose i}}{p}(1-\lambda_0^i)z^{2p-i}+\sum_{i=1}^{p-1}(-1)^i\frac{{p \choose i}}{p}(\lambda_0^i-\lambda_0^p)z^{p-i}-\lambda_1(z^p-1)}{(1-\lambda_0^p)z^p}
$$

So one has
$$
\zeta_{\alpha}(d\log z\otimes 1)=d\log z; \quad \zeta_{\beta}(d\log z\otimes 1)=d\log z+da,
$$
and 
$$
dh_{\alpha\beta}(d\log z\otimes 1)=(\zeta_{\beta}-\zeta_{\alpha})(d\log z)=da.
$$ 
Here $\xi_{\alpha},\xi_{\beta}$ and $h_{\alpha\beta}$ appear in Deligne-Illusie's Lemma (see Lemma 2.1 \cite{LSZ0}).

On the other hand, our Higgs bundle reads
$$
E:=E_{unif}=\mathcal{O}(1)\oplus \mathcal{O}(-1). 
$$  
Set 
$$
E_{\alpha}=E|_{U_{\alpha}}=\mathcal{O}_{U_{\alpha}}\{e^1_{\alpha},e^0_{\alpha}\},\quad 
E_{\beta}=E|_{U_{\beta}}=\mathcal{O}_{U_{\beta}}\{e^1_{\beta},e^0_{\beta}\}
$$
and the transition is given by
$$
\{e^1_{\beta},e^0_{\beta}\}=\{e^1_{\alpha},e^0_{\alpha}\}\left(\begin{array}{cc}
z-1& 0 \\
0 & (z-1)^{-1} \\
\end{array}\right).
$$
(The reason is as follows: $\Omega^1=\mathcal{O}(-2)$ and over $U_{\alpha}$ it has basis $e_{\alpha}=d(\frac{z-\lambda_0}{z-1})$ and over $U_{\beta}$ it has basis $e_{\beta}=dz$, the transition is given $e_{\beta}=\frac{(z-1)^2}{\lambda_0-1}e_{\alpha}$.) Therefore, $H$ is obtained by gluing 
$$
H_{\alpha}=F^*E_{\alpha},\quad H_{\beta}=F^*E_{\beta}
$$
via the gluing matrix
$$
\{e^1_{\beta}\otimes 1,e^0_{\beta}\otimes 1\}=\{e^1_{\alpha}\otimes 1,e^0_{\alpha}\otimes 1\}\left(\begin{array}{cc}
(z-1)^{p}& 0 \\
0 & (z-1)^{-p} \\
\end{array}\right)\left(\begin{array}{cc}
1& 0 \\
a & 1 \\
\end{array}\right).
$$
Note $\mathcal{O}_{U_{\alpha}}=k[\frac{z-\lambda_0}{z-1},\frac{z-1}{z-\lambda_0}]$ and $\mathcal{O}_{U_{\beta}}=k[z,z^{-1}]$. We are computing some matrices $P\in \mathrm{GL}_2(\mathcal{O}_{U_{\alpha}})$ and $Q\in \mathrm{GL}_2(\mathcal{O}_{U_{\beta}})$ such that 
$$
P\cdot \left(\begin{array}{cc}
(z-1)^{p}& 0 \\
a(z-1)^{-p} & (z-1)^{-p} \\
\end{array}\right)\cdot Q
$$
is diagonal.

Notice that Proposition \ref{condition for periodicity in char p} is only a special case of the following statement. 
\begin{prop}\label{calculation of n}
Let $\lambda=(\lambda_0,\lambda_1)$ and the resulting $H$ be as above. Then $H\cong \sO(n)\oplus \sO(-n)$ if and only if $T_n$ is the first full rank matrix in the sequence of matrices $\{T_0,T_1,\cdots,T_{p-1}\}$. 
\end{prop}
\begin{proof}
	We show that the transition matrix can be diagonilized to $\mathrm{diag}((z-1)^n,(z-1)^{-n})$ by the following algorithm.

	Denote $a$ by $A/z^p$, and notice that $A$ has degree $2p-1$. Our following argument actually does not require this form to be simplified.
	
	Step 1: Find $f,g\in k[z]$, $\mathrm{deg}(f),\mathrm{deg}(g)\leq p$ such that $f\cdot A+g\cdot z^p$ is divisible by $(z-1)^{2p}$ and $(f,g)=(z-1)^l$ for some $l\geq 0$ by the following algorithm:
	
	Consider the following equations, where $Q_i,R_i\in k[z]$, $\mathrm{deg}(R_i)\leq 2p-1$:
	\begin{align*}
	A=Q_0\cdot(z-1)^{2p}+R_0 \\
	z\cdot A=Q_1\cdot(z-1)^{2p}+R_1 \\
	\cdots \\
	z^p\cdot A=Q_p\cdot(z-1)^{2p}+R_p
	\end{align*}
	and denote the coefficient of $z^j$ in $R_i$ by $R_{ij}$.
	Since the matrix $(R_{ij})_{0\leq i \leq p, 0\leq j \leq p-1}$ has at most rank $p$, the following linear system has a non-zero solution:
	\begin{gather}
	\begin{pmatrix}
	f_0 & \cdots & f_p
	\end{pmatrix}
	\cdot
	\begin{pmatrix}
	R_{00} & \cdots & R_{0,p-1}\\
	\vdots & \ddots & \vdots \\
	R_{p,0} & \cdots & R_{p,p-1}
	\end{pmatrix}
	=
	\begin{pmatrix}
	0 & \cdots & 0
	\end{pmatrix}
	\end{gather}
	And by taking $f=f_pz^p+\cdots+f_0$, we have $f\cdot A=h\cdot (z-1)^{2p}+R$, where $h=f_p\cdot Q_p+\cdots+f_0\cdot Q_0$, $R=R_p\cdot f_p+\cdots+R_0\cdot f_0$. By the matrix above, we know that the coefficient of $1,z,\cdots,z^{p-1}$ of $R$ are all $0$, thus $R$ can be written as $-g\cdot z^p$, and $\mathrm{deg}(g)\leq p-1$. After we find a pair of $f$ and $g$, it is not difficult to find a pair that satisfies $(f,g)=(z-1)^l$.
	
	Step 2: Denote $c=\mathrm{max}(\mathrm{deg}(f),\mathrm{deg}(g))$, and $h$ as above. Find $\beta,\gamma\in k[z]$ such that $\text{deg}(\beta),\text{deg}(\gamma)\leq c-2p$ and $f\cdot\gamma+g\cdot\beta=(z-1)^{2p}$. Take $\bar{f}=f/(z-1)^l$, $\bar{g}=g/(z-1)^l$, we have $(\bar{f},\bar{g})=1$.
	
	Case A: $z-1\nmid \bar{f}$. In this case, find $\sigma\in k[z]$ such that $\text{deg}(\sigma)\leq 2p-2c+l$ and $(z-1)^l | h(\beta-\sigma\cdot\bar{f})-z^p$.
	
	Case B: $z-1 | \bar{f}$. In this case, we must have $z-1\nmid\bar{g}$, thus we can find $\sigma\in k[z]$ such that $\text{deg}(\sigma)\leq 2p-2c+l$ and $(z-1)^l | A-(\gamma+\sigma\cdot\bar{g})h$.
	
	Take $\gamma'=\gamma+\sigma\cdot\bar{g}$,$\beta'=\beta-\sigma\cdot\bar{f}$. We still have $\text{deg}(\beta'),\text{deg}(\gamma')\leq 2p-c$ and $f\cdot\gamma'+g\cdot\beta'=(z-1)^{2p}$. Take $\alpha=(a\cdot\beta'-\gamma')/(z-1)^{2p}$. Since $\alpha=(A\cdot\beta'-z^p\cdot\gamma')/z^p(z-1)^{2p}$, from the fact that $f\cdot(A\cdot\beta'-z^p\cdot\gamma')=(h\beta'-z^p)(z-1)^{2p}$ and $g\cdot(A\cdot\beta'-z^p\cdot\gamma')=(A-\gamma' h)(z-1)^{2p}$ we know that $A\cdot\beta'-z^p\cdot\gamma'$ is divisible by $(z-1)^{2p}$. So we have $\alpha\in k[z,1/z]$.
	
	By direct calculation, the following equations holds:
	\begin{gather}
	\begin{pmatrix}
	\alpha & \beta' \\
	-\frac{h}{z^p} & f
	\end{pmatrix}
	\cdot
	\begin{pmatrix}
	(z-1)^p & 0 \\
	a(z-1)^{-p} & (z-1)^{-p}
	\end{pmatrix}
	\cdot
	\begin{pmatrix}
	f\cdot(z-1)^{-c} & -\beta'\cdot(z-1)^{c-2p} \\
	g\cdot(z-1)^{-c} & \gamma'\cdot(z-1)^{c-2p}
	\end{pmatrix}
	\\
	=
	\begin{pmatrix}
	(z-1)^{p-c} & 0 \\
	0 & (z-1)^{c-p}
	\end{pmatrix}
	\end{gather}
	Thus, we have $H\cong\sO(p-c)\oplus\sO(c-p)$.
	
	Notice that we have $\mathrm{deg}(g)\leq p-1$, by taking the remainder of $(z-1)^{2p}$ on both side of $f\cdot A+g\cdot z^p=h\cdot (z-1)^{2p}$ we get the following equation:
	$$
	R_p\cdot f_p+\cdots+R_0\cdot f_0+g\cdot z^p=0
	$$
	Or we can write this equation in the following form:
	\begin{gather}
	\begin{pmatrix}
	f_0 & \cdots & f_p
	\end{pmatrix}
	\cdot
	\begin{pmatrix}
	R_{00} & \cdots & R_{0,2p-1}\\
	\vdots & \ddots & \vdots \\
	R_{p,0} & \cdots & R_{p,2p-1}
	\end{pmatrix}
	+
	\begin{pmatrix}
	0 & \cdots & 0 & g_0 & \cdots & g_{p-1}
	\end{pmatrix}
	=
	\begin{pmatrix}
	0 & \cdots & 0
	\end{pmatrix}
	\end{gather}
	Since we have the precise expresssion of $a$, we can calculate the value of elements of $R$ and we have $R_{ij}=T_{i+1,j+1}$ for $i,j\leq p-1$ and $R_{ij}=T_{i+1,3p-j}$ for $i\leq p-1, 2p-i\leq j \leq 2p-1$.
	
	We can see that $c=\mathrm{max}(\mathrm{deg}(f),\mathrm{deg}(g))$ if and only if the following condition holds: 
	
	First, the following linear system has a non-zero solution:
	\begin{gather}
	\begin{pmatrix}
	f_0 & \cdots & f_c
	\end{pmatrix}
	\cdot
	\begin{pmatrix}
	R_{00} & \cdots & R_{0,p-1} & R_{0,2p-1} & \cdots & R_{0,p+c+1}\\
	\vdots & \ddots & \vdots & \vdots & \ddots & \vdots \\
	R_{c,0} & \cdots & R_{c,p-1} & R_{c,2p-1} & \cdots & R_{c,p+c+1}
	\end{pmatrix}
	=
	\begin{pmatrix}
	0 & \cdots & 0
	\end{pmatrix}
	\end{gather}
	(Notice that the $(c+1)\times(2p-c-1)$ matrix is exactly $T_{p-c-1}$.)
	
	Meanwhile, the following linear system does not have a non-zero solution:
	\begin{gather}
	\begin{pmatrix}
	f_0 & \cdots & f_{c-1}
	\end{pmatrix}
	\cdot
	\begin{pmatrix}
	R_{00} & \cdots & R_{0,p-1} & R_{0,2p-1} & \cdots & R_{0,p+c}\\
	\vdots & \ddots & \vdots & \vdots & \ddots & \vdots \\
	R_{c-1,0} & \cdots & R_{c-1,p-1} & R_{c-1,2p-1} & \cdots & R_{c-1,p+c}
	\end{pmatrix}
	=
	\begin{pmatrix}
	0 & \cdots & 0
	\end{pmatrix}
	\end{gather}
	(Notice that the $c\times(2p-c)$ matrix is exactly $T_{p-c}$.)
	
	Thus, for the sequence $\{T_0,\cdots,T_{p-1}\}$, $T_{p-c-1}$ (and all the matrices before $T_{p-c-1}$) cannot be of full rank and $T_{p-c}$ (and all the matrices after $T_{p-c}$ is (are) of full rank.

\end{proof}

\section{Semistable families of abelian varieties with real multiplication}\label{construction of families}

In this section, we shall construct a semistable family of abelian varieties with $\sO_F$ multiplication. The input are the properties (i)-(iii) in Theorem \ref{main result}. We denote $\P^1-\{0,1,\lambda,\infty\}$ by $X$.

Assume (i). Let $[F:\Q]=d$. Let $\V_{\sO_F}$ be the corresponding $\sO_F$ local system of rank 2 to $\rho_{\lambda}$. We define 
$$
\omega:\sO_F^{\oplus2}\times\sO_F^{\oplus2}\xrightarrow{\det}\sO_F\xrightarrow{\textrm{Tr}}\Z
$$
\begin{lemma}
$\omega$ is skew-symmetric and non-degenerate.
\end{lemma}
Assume (ii). Regarding $\V_{\sO_F}$ as a $\Z$-local system, it underlies a polarized $\Z$-variation of Hodge structure (with polarization $\omega$). Let $H=\V_{\sO_F}\otimes_{\Z}\sO_X$ and $Fil^1\subset H$ be the Hodge filtration. As the composite $\V_{\sO_F}\hookrightarrow H\rightarrow H/\mathrm{Fil}^1$ is injective, we get a family of abelian varieties 
$$
f: \mathcal{A}:=\V_{\sO_F}\backslash H/Fil^1\to X.
$$ 
We compactify it to a projective morphism $\bar f: \bar \sA\to \P^1$. By blowing-up with centers contained in the singular fibers of $\bar f$, we may make $\bar f$ into a quasi-semistable family. Since the local monodromies of $\V_{\sO_F}$ around $D$ are unipotent, we actually obtain a semistable family $\bar f$. 

Assume (iii). We show that the family $f$ admits $\sO_F$ multiplication, that is, $\sO_F\subset \End_\Z(\sA/X)$. By the assumption, there exists one point $x\in X$ such that the multiplication by any element of $\sO_F$ is of Hodge type $(0,0)$. It is clear that the multiplication by $\sO_F$ on $\V_{\sO_F}$ commutes with the monodromy action. So the multiplication by an element of $\sO_F$ is a global section of the weight zero $\Z$-VHS $\End_\Z(\V_{\sO_F})\cong \V_{\sO_F}^*\otimes_{\Z}\V_{\sO_F}$. Applying the rigidity theorem \cite[Corollary 7.23]{Sc}, the multiplication by any element of $\sO_F$ is of Hodge type $(0,0)$ everywhere, which means $\sO_F\subset\End_\Z(\mathcal{A}/X)$. The construction is completed.

\section{Proof of main result}
For any $\lambda$ in the Beauville's list, Beauville constructs a semistable family $\bar f_{\lambda}$ of elliptic curves over $\P^1$ with singular locus exactly equal to $\{0,1,\lambda,\infty\}$. Let $f_{\lambda}$ be the smooth part of $\bar f_{\lambda}$. We shall explain that the uniformizing representation $\rho_{\lambda}$ over $P^1-\{0,1,\lambda,\infty\}$ satisfies the properties in Theorem \ref{main result}. We simply take $F=\Q$ and claim that the monodromy representation associated to the family $f_{\lambda}$ underlies the real local system $\rho_{\lambda}$. If the claim holds, then the properties (i)-(iii) become obvious. The most efficient way to get this is to use the Simpson correspondence \cite{Si}. Set $\rho_{geo}$ to be the monodromy representation of $f_{\lambda}$. The corresponding logarithmic Higgs bundle to $\rho_{geo}\otimes_{\Z}\C$ is the logarithmic Kodaira-Spencer system $(E,\theta)$ associated to $f_{\lambda}$. The proof of Proposition \ref{beauville's list implies periodicity} explains that $(E,\theta)\cong (E_{unif},\theta_{unif})$. Therefore, $\rho_{geo}\otimes_{\Z}\C\cong \rho_{\lambda}\otimes_{\R} \C$. Set $\rho_{geo,\R}=\rho_{geo}\otimes \R$. So we have
$$
H^0(U, \rho_{geo,\R}\otimes_{\R}\rho^{*}_{\lambda})\otimes_{\R}\C=H^0(U, (\rho_{geo,\R}\otimes_{\R}\rho^{*}_{\lambda})\otimes_{\R}\C)
$$      
is nonzero. Thus, $H^0(U, (\rho_{geo,\R}\otimes_{\R}\rho^{*}_{\lambda}))\cong \R$ as the complex local systems are irreducible. It follows that $\rho_{geo}\otimes_{\Z}\R\cong \rho_{\lambda}$ as claimed. 

The difficult part of Theorem \ref{main result} is the converse direction. By the construction of semistable families of abelian varieties in \S\ref{construction of families} and the classification of Beauville \cite{Be}, it suffices to show that properties (i)-(iii) forces $F$ to be $\Q$. In other words, we are going to prove the following negative result:
\begin{claim}
For any $\lambda$, there exists \emph{no} totally real subfield $F$ of degree $>1$ such that $\rho_{\lambda}$ satisfies properties (i)-(iii) simultaneously. 
\end{claim}
\begin{proof}
Let $\bar f: \sA\to \P^1$ be the semistable family of abelian varieties with singular locus $D$ and with real multiplication $\sO_F$. Let $\rho$ be the weight one $\Z$-VHS associated to $f$, and $(E,\theta)$ be the corresponding logarithmic Kodaira-Spencer system. By construction, $\rho=Res_{F|\Q}\rho_{\lambda}$ and hence $\rho_{\lambda}\subset \rho\otimes \R$ as $\R$-sublocal system. So $(E_{unif},\theta_{unif})\subset (E,\theta)$ and by Proposition \ref{polystable}, it is in fact a direct factor. We know by Theorem \ref{periodicity theorem}, $(E_{unif},\theta_{unif})$ is periodic. The contradiction arises from the exact analysis of the period. First of all, by Proposition \ref{one periodicity}, the period of $(E_{unif},\theta_{unif})$ is in any case equal to one. However, we are going to argue that the existence of the real multiplication $\sO_F$ ($d=[F:\Q]>1$) as endomorphism of $f$ leads to a contradiction with one-periodicity.

By \v{C}ebotarev density theorem, the set of inert primes of $F$ is of positive Dirichlet density. In particular, it is an infinite set. Let $\mathfrak{p}$ be an inert prime over the rational prime $p$. So there are natural isomorphisms of $\Z_p$-algebras:
$$
\sO_F\otimes_{\Z}\Z_p\cong \sO_{F_{\mathfrak p}}\cong \Z_{p^d},
$$
where $F_{\mathfrak p}$ is the completion of $F$ at the prime $\mathfrak p$. Let $\mathfrak f$ be a spread of $f$ defined over $S$ and let $(\sE,\Theta)$ be the corresponding logarithmic Kodaira-Spencer system (see the proof of Theorem \ref{periodicity theorem}). Let $s\in S$ be a geometrically closed point of $\textrm{char}(k(s))=p$ (which is assumed to be large enough) and $\tilde s: \Spec(W(k(s)))\to S$ a closed subscheme lifting $s: \Spec(k(s))\to S$. Let $\hat s$ be the generic point of $\tilde s$. Set $\mathfrak{f}_{\tilde s}$ to be base change of $\mathfrak f$ over $\tilde s$. It is a semistable family over $W(k(s))$ with $\sO_F\subset \End(\mathfrak{f}_{\tilde s})$.

Let $\V$ be the mod $p$ crystalline representation of the algebraic fundamental group $\pi_1(U_{\hat s})$ of the family $\mathfrak{f}_{\tilde s}$, which corresponds to the one-periodic flow with the initial Higgs term $(\sE_{s},\Theta_{s})$, the logarithmic Kodaira-Spencer system associated to the special fiber of $\mathfrak{f}_{\tilde s}$ (Theorem 1.1 \cite{LSYZ}). Now by transportation of structure, the local system $\V$ contains $\sO_F$ as endomorphism subalgebra. Thus, $\V$ is a $\pi_1(U_{\hat s})-\sO_F\otimes_{\Z}\F_p$-module and by the previous discussion, it means that $\V$ is the restriction of scalar of a $\F_{p^d}$-local system, that is, $\V=Res_{\F_{p^d}|\F_p}\W$ for some rank two $\F_{p^d}$-crystalline representation $\W$. By Corollaries 3.10, 3.11 (ii) \cite{LSZ}, it follows that any simple factor of $(\sE_{s},\Theta_{s})$ is $d$-periodic and all its simple factors are pairwisely non-isomorphic. In particular, the factor $(E_{unif},\theta_{unif})$ cannot be one-periodic. Contradiction. The claim is proved.

\end{proof}


\begin{thebibliography}{X-X00}
\addcontentsline{toc}{chapter}{Bibliography}

\bibitem[Be]{Be}
A. Beauville, Les familles stables de courbes elliptiques sur $\P^1$ admettant 4 fibres singuli\`{e}res. C. R. Acad. Sc. Paris 294 (1982), 657-660 

\bibitem[DI]{DI}
P. Deligne, L. Illusie, Rel\`{e}vements modulo $p^2$ et decomposition du complexe de de Rham, Invent. Math. 89 (1987), 247-270.

\bibitem[Fa]{Fa}
G. Faltings, Real projective structures on Riemann surfaces, Compositio Mathematicae, tome 48, no 2 (1983), p. 223-269.

\bibitem[Fa89]{Fa89}
G. Faltings, Crystalline cohomology and $p$-adic Galois-representations, Algebraic analysis, geometry, and number theory (Baltimore, MD, 1988), 
25-80
 

\bibitem[KYZ]{KYZ}
R. Krishnamoorthy, J.-B. Yang, K. Zuo, A $p$-adic characterization of modular curves. Preprint. 

 

\bibitem[IL90]{IL90}
L. Illusie, R\'eduction semistable et d\'ecomposition de complexes de de Rham
\`a\ coefficients, Duke Math. J., vol. 60, no. 1 (1990), 139-185.



\bibitem[LSYZ]{LSYZ}G. Lan, M. Sheng, Y. Yang, K. Zuo, Uniformization of $p$-adic curves via Higgs-de Rham flows, J. Reine Angew. Math. 747 (2019), 63-108.



\bibitem[LSZ15]{LSZ0} G. Lan, M. Sheng, K. Zuo, Nonabelian Hodge theory via exponential twisting, Math. Res. Lett. 22 (2015), no. 3, 859-879.


\bibitem[LSZ]{LSZ} G. Lan, M. Sheng, Y. Yang, K. Zuo, Semistable Higgs bundles, periodic Higgs bundles and representations of algebraic fundamental groups, J. Eur. Math. Soc.  Volume 21, Issue 10 (2019), 3053-3112


\bibitem[Sc]{Sc}
W. Schmid, Variation of Hodge Structure: The Singularities of the Period Mapping, Inventiones math. 22 (1973), p. 211-319

\bibitem[Si]{Si}
C. Simpson, Harmonic bundles on noncompact curves, J. Amer.
Math. Soc. 3, no. 3 (1990), 713-770.

\bibitem[Si92]{Si92}
C. Simpson, Higgs bundles and local systems, Inst. Hautes \'{E}tudes Sci. Publ.
Math. No. 75 (1992), 5-95.

\bibitem[Zuo]{Zuo}
K. Zuo, Notes on arithmetic Simpson correspondence and $GL_2$-motivic local systems over $\P^1-\{0,1,\lambda,\infty \}$, Preprint, 2018.

\bibitem[EGA IV]{EGA IV}
A. Grothendieck and J. Dieudonn\'{e}, \'{E}tude locale des sch\'{e}mas et des morphisms de sch\'{e}mas, Publ. Math. IHES, 20(1964), 24(1965), 28(1966), 32(1967).


 
\end{thebibliography}
\end{document}